\documentclass{amsart}
\usepackage{graphicx} 
\usepackage{amsfonts}
\usepackage{amssymb}
\usepackage{color}
\usepackage{amsmath}
\usepackage{amsthm}
\usepackage{multicol}
\usepackage{mathtools}
\usepackage{tikz,tikz-cd, tcolorbox}
\usepackage[colorlinks=true,linkcolor=blue,citecolor=blue]{hyperref}
\usepackage{setspace, mathrsfs}
\usepackage{bussproofs}
\theoremstyle{plain}
\newtheorem{theorem}{Theorem}[section]
\newtheorem{lemma}[theorem]{Lemma}

\usepackage{amsaddr}

\theoremstyle{definition}
\newtheorem{definition}[theorem]{Definition}
\newtheorem{remark}[theorem]{Remark}

\begin{document}
\title[Duality for partial orthomodular lattices]{Duality for partial orthomodular lattices}
\author{Joseph McDonald }
\address{Institute of Computer Science, Czech Academy of Sciences}

\email{mcdonald@cs.cas.cz}
\date{May 2026}
\begin{abstract}
  Partial orthomodular lattices are an intermediate class of algebras between ortholattices and orthomodular lattices. In this short note, we obtain a duality for partial orthomodular lattices via a subcategory of the category of spectral spaces, independently of the Axiom of Choice. 
\end{abstract}

\maketitle

\section{Introduction}
Partial orthomodular lattices (pOMLs) were introduced in \cite{leem} as an intermediate class of algebras between the class of ortholattices (OLs) and the class of orthomodular lattices (OMLs), and capture important algebraic aspects of various geometric structures in which the orthomodularity identity fails, such as in the lattice of closed convex cones of $\mathbb{R}$ and $\mathbb{R}^2$. In this short note, we extend the relational representation results of pOMLs given in \cite{leem} to a duality between the category of pOMLs and a certain subcategory of the category of spectral spaces. This is achieved as a certain restriction to the duality results obtained in \cite{mcdonald} for OLs, which generalizes the upper Vietoris space duality for Boolean algebras developed in \cite{bez}. In \cite{mcdonald}, an arbitrary OL is represented via the compact open biorthogonally closed subsets of a spectral space, known as an \emph{upper Vietoris orthospace}, endowed with the underlying first-order relational structure of an \emph{orthoframe}. This is a set equipped with a binary relation that is irreflexive and symmetric. In this work, an arbitrary pOML is represented via the compact open biorthogonally closed subsets of an upper Vietoris orthospace whose underlying relational structure is restricted to that of a $c$-frame. This is an orthoframe extended by a second-order condition quantifying over certain biorthogonally closed subsets. A noteworthy feature of the duality provided here (as well as in \cite{mcdonald}) is that unlike the duality presented for general OLs in \cite{goldblatt1,bimbo} using Stone spaces, our spectral duality for pOMLs obtains in ZF independently of the Axiom of Choice.

\section{Topological representation and duality}
Recall that an OL is a bounded lattice $\langle A;\wedge,\vee,0,1\rangle$ equipped with an order-inverting period two complementation $^\perp\colon A\to A$, known as an \emph{orthocomplementation}. An OML is an OL satisfying $a\leq b\Rightarrow b=a\vee(a^{\perp}\wedge b)$. 
\begin{definition}
    A \emph{partial orthomodular lattice} (\emph{pOML}) is an OL such that  $b\leq a$ and $a=b\vee(a\wedge b^{\perp})$ implies $b=a\wedge(a^{\perp}\vee b)$.  
\end{definition}

   Note that not every OL is a pOML but every OML is a pOML (see \cite[Propositions 2 and 3]{leem}).                                    An \emph{orthoframe} is a pair $\langle X;\perp\rangle$ such that $X$ is a set and $\perp\subseteq X^2$ is irreflexive and symmetric. For any orthoframe $X$ and $U\subseteq X$, let \[U^{\perp}:=\{x\in X:x\perp y\hspace{.1cm}\text{for all}\hspace{.1cm}y\in U\}.\] In \cite{goldblatt2} a subset $U\subseteq X$ is defined as $\perp$-\emph{closed} relative to a subset $V\subseteq X$ if for all $x\in V$, if 
    $x\not\in U$, there exists $y\in V$ such that $y\perp U$ and $x\not\perp y$. We denote the sets that are $\perp$-closed relative to the entire space $X$ by $\mathcal{B}(X)$. Note that $\mathcal{B}(X)$ is precisely those subsets that are biorthogonally closed in the sense that $\mathcal{B}(X)=\{U\subseteq X:U=U^{\perp\perp}\}$.     
\begin{definition}
    A $c$\emph{-frame} is a triple $\langle X;\perp,\Omega(X)\rangle$ such that $\langle X;\perp\rangle$ is an orthoframe and $\Omega(X)\subseteq\mathcal{B}(X)$ is closed under $\cap$, $^{\perp}$, and satisfies the following condition: if $V\subseteq U$ and $U^{\perp}$ is $\perp$-closed in $V^{\perp}$, then $V$ is $\perp$-closed in $U$.
\end{definition}
\begin{lemma}\label{lemma 2.3}
    If $X$ is a c-frame, then $U=V\cap(V^{\perp}\sqcup U)$ implies that $U$ is $\perp$-closed in $V$ for all $U,V\in\Omega(X)$ where $U\sqcup V:=(U\cup V)^{\perp\perp}$.  
\end{lemma}
\begin{proof}
    Take $x\in V$ with $x\not\in U$. Since $U=V\cap(V^{\perp}\sqcup U)$, we have $x\not\in V\cap(V^{\perp}\sqcup U)$. Moreover, since $x\in V$, it follows that $x\not\in V^{\perp}\sqcup U=(V\cap U^{\perp})^{\perp}$ and thus there exists some $y\in V\cap U^{\perp}$ such that $x\not\perp y$. Hence $y\in V$ and $y\in U^{\perp}$ which implies $y\perp U$. Therefore $U$ is $\perp$-closed in $V$.
    \end{proof}
\begin{lemma}\label{c-frame to poml}
    If $X$ is a c-frame, then $\langle\Omega(X);\cap,\sqcup,^{\perp},\emptyset,X\rangle$ is a pOML. 
\end{lemma}
    \begin{proof}
Since $\Omega(X)\subseteq\mathcal{B}(X)$, it follows that $\langle\Omega(X);\cap,\sqcup,^{\perp},\emptyset,X\rangle$ forms an OL. It remains to verify that it is a pOML. Assume $U,V\in\Omega(X)$ with $V\subseteq U$ and $U=V\sqcup(U\cap V^{\perp})$. We first show $U^{\perp}$ is $\perp$-closed in $V^{\perp}$. Since $^{\perp}$ is an orthocomplementation, De Morgan's identities give $U^{\perp}=V^{\perp}\cap(V\sqcup U^{\perp})$. Then Lemma \ref{lemma 2.3} implies $U^{\perp}$ is $\perp$-closed in $V^{\perp}$. Since $X$ is a c-frame, it follows that $V$ is $\perp$-closed in $U$. To show $V=U\cap(U^{\perp}\sqcup V)$, we demonstrate the $U\cap(U^{\perp}\sqcup V)\subseteq V$ inclusion as the other inclusion is trivial. Hence assume $x\not\in V$. Since $V$ is $\perp$-closed in $U$, there exists $y\in U$ such that $x\not\perp y$ and $y\perp V$. Moreover, we have $y\in U\cap V^{\perp}$. Now suppose for the sake of contradiction that $x\in U\cap(U^{\perp}\sqcup V)$ i.e., $x\in U\cap(U\cap V^{\perp})^{\perp}$. Then $x\perp z$ for every $z\in U\cap V^{\perp}$ but this contradicts our hypothesis that $x\not\perp y$ for some $y\in U\cap V^{\perp}$. Hence $x\not\in U\cap(U^{\perp}\sqcup V)$ so $U\cap(U^{\perp}\sqcup V)\subseteq V$.  
\end{proof} 
\begin{definition}
    Let $A$ be a pOML. Then define $\mathcal{F}(A)=\langle X_A;\perp_A,\beta_A,\tau_A\rangle$ by: 
    \begin{enumerate}
        \item $X_A$ is the collection of all non-empty proper filters of $A$; 
        \item $x\perp_A y$ iff there exists $a\in A$ such that $a\in x$ and $a^{\perp}\in y$; 
        \item $\beta_A=\{\phi(a):a\in A\}$ where $\phi(a)=\{x\in X_A:a\in x\}$; 
        \item $\tau_A$ is the topology on $X_A$ generated by the basis $\beta_A$. 
    \end{enumerate}
\end{definition}
For a topological space $X$, let $\mathcal{C}(X)$ be the compact subsets of $X$, let $\mathcal{O}(X)$ be the open subsets of $X$, and let $\mathcal{K}(X):=\mathcal{C}(X)\cap\mathcal{O}(X)$. Recall that $X$ is a \emph{spectral space} if $X$ is compact, $T_0$, coherent, and sober (see \cite[Section 3.2]{mcdonald}). Since every spectral space $X$ is $T_0$, one can define a specialization order by $x\leqslant y$ iff $x\in U$ implies $y\in U$ for all $U\in\mathcal{O}(X)$, that is a partial order.  
\begin{lemma}\label{spectral space}
     $\langle X_A;\tau_A\rangle$ is a spectral space whose specialization order is $\subseteq$.  
\end{lemma}
\begin{proof}
Since every pOML is an OL, the result is an immediate application of \cite[Proposition 3.4.1]{mcdonald}.  
\end{proof}
\begin{remark}
    It was shown in \cite[Proposition 3]{goldblatt1} that for any OL $A$, the Stone space $\langle X_A;\tau_A^*\rangle$ is compact where $\tau_A^*\subseteq\wp(X_A)$ is the topology generated by the subbasis $\bigcup_{a\in A}\{\phi(a),\hspace{.1cm}X_A\setminus\phi(a)\}$. By \cite[Proposition 3.3.1]{mcdonald}, the compactness of $\langle X_A;\tau_A^*\rangle$ requires the Axiom of Choice. As the proof of compactness of $\langle X_A;\tau_A\rangle$ follows by \cite[Proposition 3.4.1]{mcdonald}, it obtains in ZF alone.     
\end{remark}
\begin{lemma}\label{lemma 2.8}
    In any OL, we have $b\wedge(b^{\perp}\vee a)\leq a$ iff $\phi(a)$ is $\perp$-closed in $\phi(b)$. 
\end{lemma}
\begin{proof}
    This is an application of \cite[Lemma 6.3]{goldblatt2}, \cite[Theorem 3.7]{goldblatt2}, and the fact that OLs provide sound and complete algebraic models for orthologic.      
\end{proof}
\begin{lemma}\label{poml to c-frame}
    $\langle X_A;\perp_A,\beta_A\rangle$ is a c-frame. 
\end{lemma}
\begin{proof}
    The proof that $\langle X_A;\perp_A\rangle$ is a orthoframe follows from \cite[Proposition 2]{goldblatt1}. It suffices to show that for all $a,b
    \in A$, if $\phi(b)\subseteq\phi(a)$ and $\phi(a)^{\perp}$ is $\perp$-closed in $\phi(b)^{\perp}$, then $\phi(b)$ is $\perp$-closed in $\phi(a)$. First note that the definition of $\phi$ applied to our hypothesis that $\phi(b)\subseteq\psi(a)$ gives $b\leq a$. Applying Lemma \ref{lemma 2.8} to our hypothesis that $\phi(a)^{\perp}$ is $\perp$-closed in $\phi(b)^{\perp}$ yields $b^{\perp}\wedge(b^{\perp\perp}\vee a^{\perp})\leq a^{\perp}$. Since $^{\perp}$ is an order-inverting involution, $a^{\perp\perp}=a\leq(b^{\perp}\wedge(b^{\perp\perp}\vee a^{\perp}))^{\perp}$. De Morgan's identities and the fact that $^{\perp}$ is an involution then guarantees that  \[(b^{\perp}\wedge(b^{\perp\perp}\vee a^{\perp}))^{\perp}=b\vee(b^{\perp}\wedge a)\] so $a\leq b\vee(b^{\perp}\wedge a)$. Clearly $b^{\perp}\wedge a\leq a$ which together with the fact that $b\leq a$ and the anti-symmetry of $\leq$ yields $a=b\vee(b^{\perp}\wedge a)$. Since $A$ is a pOML, we have $b=a\wedge(a^{\perp}\vee b)$ so $a\wedge(a^{\perp}\vee b)\leq b$. Lemma \ref{lemma 2.8} implies $\phi(b)$ is $\perp$-closed in $\phi(a)$.  
\end{proof}
For a $c$-frame $X$ equipped with a topology, we define \[\mathcal{K}\Omega(X):=\mathcal{K}(X)\cap\Omega(X).\] 
\begin{theorem}\label{isomorphism}
    Every pOML $A$ is isomorphic to $\mathcal{K}\Omega(\mathcal{F}(A))$.  
\end{theorem}
\begin{proof}
    The map $\phi(a)=\{x\in X_A:a\in x\}$ provides an isomorphism. The proof is a simple application of \cite[Theorem 3.4.2]{mcdonald}, however we outline the proof that $\phi$ is a bijection. For any $c\in A$, let ${\uparrow}c=\{d\in A:c\leq d\}$. Clearly ${\uparrow}c$ is a proper filter whenever $c\not=0$. For injectivity, assume $a\not\leq b$. Then $a\in{\uparrow}a$ but $b\not\in{\uparrow}a$ so ${\uparrow}a\in\phi(a)$ but ${\uparrow}a\not\in\phi(b)$ and hence $\phi(a)\not\subseteq\phi(b)$. For surjectivity, suppose $U\in\mathcal{K}\Omega(\mathcal{F}(A))$ so that $U=\bigcup^n_{i=1}\phi(a_i)$ for $a_1,\dots a_n\in A$. Since $U\in\Omega(\mathcal{F}(A))$, it can be easily shown that $U=U^{\perp\perp}$ and hence we obtain \[U=U^{\perp\perp}=\biggl(\bigcup_{i=1}^n\phi(a_i)\biggl)^{\perp\perp}=\bigvee^n_{i=1}\phi(a_i)=\phi\biggl(\bigvee_{i=1}^na_i\biggl).\] Therefore $U$ is in the image of $\phi$ and hence $\phi$ is surjective. 
\end{proof}
 Below are a restriction to the upper Vietoris orthospaces (the choice-free duals of the OLs) introduced in \cite{mcdonald}. For the following definition, let \[\mathcal{K}\Omega_X(x):=\{U\in\mathcal{K}\Omega(X):x\in U\}.\]  
\begin{definition}\label{uvc-space}
    An \emph{upper Vietoris c-space} (\emph{UVC-space}) is a c-frame $X$ equipped with a $T_0$-topology $\tau\subseteq\wp(X)$ satisfying the following conditions: 
    \begin{enumerate}
        \item $\mathcal{K}\Omega(X)$ is closed under $\cap$ and $^{\perp}$ and forms a basis for $X$;
        \item if $x\perp y$, there exists $U\in\mathcal{K}\Omega(X)$ such that $x\in U$ and $y\in U^{\perp}$; 
        \item every proper filter in $\mathcal{K}\Omega(X)$ is of the form $\mathcal{K}\Omega_X(x)$ for some $x\in X$. 
    \end{enumerate}
\end{definition}
\begin{lemma}\label{functors}
    If $X$ is a UVC-space, then $\mathcal{G}(X)=\langle\mathcal{K}\Omega(X);\cap,\sqcup,^{\perp},\emptyset,X\rangle$ is a pOML. If $A$ is a pOML, then $\mathcal{F}(A)=\langle X_A;\perp_A,\beta_A,\tau_A\rangle$ is a UVC-space. 
\end{lemma}
\begin{proof}
    The first part follows by Lemma \ref{c-frame to poml} and \cite[Lemma 4.1.2]{mcdonald}. For the second part, by Lemma \ref{spectral space} and Lemma \ref{poml to c-frame}, it suffices to show that conditions 1-3 of Definition \ref{uvc-space} are satisfied. For condition 1, it follows by Theorem \ref{isomorphism} that $\phi(a)\cap\phi(b)=\phi(a\wedge b)\in\mathcal{K}\Omega(\mathcal{F}(A))$ and $\phi(a)^{\perp}=\phi(a^{\perp})\in\mathcal{K}\Omega(\mathcal{F}(A))$. Moreover, the definition of $\beta_A$ and $\tau_A$ guarantees that sets of the form $\phi(a)$ for $a\in A$ are a basis for $\mathcal{F}(A)$. For condition 2, assume $x\perp_Ay$. Then there exists some $a\in A$ such that $a\in x$ and $a^{\perp}\in y$. Therefore we have $x\in\phi(a)$ and $y\in\phi(a^{\perp})=\phi(a)^{\perp}$. For condition 3, take any proper filter $F$ in $\mathcal{K}\Omega(\mathcal{F}(A))$. Theorem \ref{isomorphism} gives that $x=\{a\in A:\phi(a)\in F\}$ is a proper filter in $A$. Therefore $x\in\mathcal{F}(A)$ and $\mathcal{K}\Omega_{\mathcal{F}(A)}(x)=F$, as desired.    
\end{proof}
\begin{theorem}\label{homeomorphism}
    Every UVC-space $X$ is homeomorphic and relationally isomorphic to $\mathcal{F}(\mathcal{G}(X))$. Moreover every UVC-space is a spectral space.  
\end{theorem}
\begin{proof}
    We show that the map $\psi(x)=\{U\in\mathcal{K}\Omega(X):x\in U\}$ exhibits the desired homeomorphism and relational isomorphism. It is an easy exercise to verify that $\psi(x)$ is a proper filter in $\mathcal{G}(X)$ for any $x\in X$ so that $\psi$ is well-defined. To see that $\psi$ is injective, note that since $\mathcal{K}\Omega(X)$ forms a basis for $X$, it follows that if $x\not\leqslant y$, there exists $U\in\mathcal{K}\Omega(X)$ such that $x\in U$ and $y\not\in Y$. The definition of $\psi$ then gives $U\in\psi(x)$ and $U\not\in\psi(y)$ so $\psi(x)\not\leqslant\psi(y)$. For surjectivity, take any $F\in\mathcal{F}(\mathcal{G}(X))$. By Definition \ref{uvc-space}(3), there exists some $x\in X$ such that $\mathcal{K}\Omega_X(x)=F$ and hence $\psi(x)=F$. To see $\psi$ is continuous, take any basic open $\phi(U)$ in $\mathcal{F}(\mathcal{G}(X))$, then 
    \begin{align*}
        \psi^{-1}[\phi(U)]&=\{x\in X:\mathcal{K}\Omega_X(x)\in\phi(U)\}\\&=\{x\in X:U\in\mathcal{K}\Omega_X(x)\}\\&=\{x\in X:x\in U\}\\&=U
    \end{align*}
   The continuity of $\psi^{-1}$ is calculated by noting that 
 \[
       \psi[\phi(U)]=\{\mathcal{K}\Omega_X(x):x\in U\}=\{\mathcal{K}\Omega_X(x):U\in\mathcal{K}\Omega_X(x)\}=\phi(U)
\]
 Lastly, it can be shown that $x\perp y$ iff $\psi(x)\perp\psi(y)$ by Definition \ref{uvc-space}(2) and that $x\leqslant y$ iff $\psi(x)\subseteq\psi(y)$ since $x\not\leqslant y$ implies there exists $U\in\mathcal{K}\Omega(X)$ such that $x\in U$ and $y\not\in U$ for any UVC-space. Hence $\psi$ is a relational isomorphism. Thus, it follows by Lemma \ref{spectral space} and Lemma \ref{functors} that $X$ is a spectral space.     
\end{proof}
 
\begin{definition}\label{uvc-map}
    Let $X$ and $Y$ be UVC-spaces. A map $f\colon X\to Y$ is a \emph{spectral c-frame map} provided: 1.) $f^{-1}[U]\in\mathcal{K}(X)$ for $U\in\mathcal{K}(Y)$, 2.) if $x\not\perp y$, then $f(x)\not\perp f(y)$, and 3.) if $z\not\perp f (y)$, there is $x\in X$ with $x\not\perp y$ and $z\leqslant f(x)$.

\end{definition}

    By $\mathbf{pOML}$ we denote the category of pOMLs and homomorphisms. By $\mathbf{UVC}$ we denote the category of UVC-spaces and spectral c-frame maps. 

\begin{theorem}
    $\mathbf{pOML}$ is dually equivalent to $\mathbf{UVC}$. 
\end{theorem}
\begin{proof}
    Let $X$ and $Y$ be UVC-spaces and $f\colon X\to Y$ be a spectral c-frame map. It is easily verified that $f^{-1}[U\cap V]=f^{-1}[U]\cap f^{-1}[V]$ and $f^{-1}[\emptyset]=\emptyset$ for any $U,V\in\mathcal{K}\Omega(X)$. Moreover, by conditions 2 and 3 of Definition \ref{uvc-map}, one can show that $f^{-1}[U^{\perp}]=f^{-1}[U]^{\perp}$ and hence $\mathcal{G}\colon f\mapsto f^{-1}$ together with the construction given in Lemma \ref{functors} determines a contravariant functor $\mathcal{G}\colon\mathbf{UVC}\to\mathbf{pOML}$. Now let $A$ and $B$ be pOMLs and let $h\colon A\to B$ be a homomorphism. We show that $\mathcal{F}\colon h\mapsto h^{-1}$ determines a spectral c-frame map from $\mathcal{F}(B)$ to $\mathcal{F}(A)$. To see that $\mathcal{F}(h)$ is spectral, take any $\phi(a)\in\mathcal{K}(\mathcal{F}(A))$, then 
    \begin{align*}
        \mathcal{F}(h)^{-1}[\phi(a)]&=(h^{-1})^{-1}[\phi(a)]\\&=\{x\in X_B:h^{-1}[x]\in\phi(a)\}\\&=\{x\in X_B:a\in h^{-1}[x]\}\\&=\{x\in X_B:h(a)\in x\}\\&=\phi(h(a))
    \end{align*}
 Since $\phi(h(a))\in\mathcal{K}(\mathcal{F}(B))$, it follows that $\mathcal{F}(h)$ is a spectral map from $\mathcal{F}(B)$ to $\mathcal{F}(A)$. Conditions 2 and 3 of Definition \ref{uvc-map} follow from \cite[Lemma 3.9]{bimbo}. This, together with the construction given in Lemma \ref{functors}, determines a contravariant functor $\mathcal{F}\colon\mathbf{pOML}\to\mathbf{UVC}$. A routine argument shows that $\mathcal{G}\colon\text{Hom}_{\mathbf{UVC}}(X,Y)\to\text{Hom}_{\mathbf{pOML}}(\mathcal{G}(Y),\mathcal{G}(X))$ and $\mathcal{F}\colon\text{Hom}_{\mathbf{pOML}}(A,B)\to\text{Hom}_{\mathbf{UVC}}(\mathcal{F}(B),\mathcal{F}(A))$ are fully faithful. Lastly the constructions of $\mathcal{G}$ and $\mathcal{F}$ together with Theorem \ref{isomorphism} and Theorem \ref{homeomorphism} provide natural isomorphisms $\phi\colon\mathbf{id}_{\mathbf{pOML}}\to\mathcal{G}\circ\mathcal{F}$ and $\psi\colon\mathbf{id}_{\mathbf{UVC}}\to\mathcal{F}\circ\mathcal{G}$ where $\mathbf{id}_{\mathbf{pOML}}$ and $\mathbf{id}_{\mathbf{UVC}}$ are the identity functors.  
\end{proof}
\noindent\textbf{Ethical approval, Data availability, Conflict of interest} Not applicable.

\noindent\textbf{Funding.} This work has been funded by a grant from the Programme Johannes Amos Comenius
under the Ministry of Education, Youth and Sports of the Czech Republic,
CZ.02.01.01/00/23$_{-}$025/0008711.


\begin{thebibliography}{99}
\bibitem{bez} Bezhanishvili, N., Holliday, W.: Choice-free Stone duality. \emph{J. Symb. Log.}, \textbf{85}, 109--148 (2020). 
\bibitem{bimbo}Bimb{\'o}, K.: Functorial duality for ortholattices and De Morgan lattices. \emph{Log. Univers.} \textbf{1}, 311--333 (2007) 
\bibitem{leem} Leemhuis, M., Wolter, D., Özçep, Ö.L.: Rules of Partial Orthomodularity. In: Metcalfe, G., Studer, T., de Queiroz, R. (eds) Logic, Language, Information, and Computation. WoLLIC 2024. Lecture Notes in Computer Science, vol 14672., 108--121, Springer, Cham. (2024)
\bibitem{goldblatt1}Goldblatt, R.I., The Stone space of an ortholattice. Bull. Lond. Math. Soc. \textbf{7}, 45--48 (1975) 
\bibitem{goldblatt2}Goldblatt, R.I., A semantic analysis of orthologic. J. Phil. Log. \textbf{3}, 19--35 (1974)
\bibitem{mcdonald} McDonald, J., Yamamoto, K.: Choice-free duality for orthocomplemented lattices by means of spectral spaces. \emph{Algebra Universalis}, vol. \textbf{83} (2022)
\end{thebibliography}
\end{document}